\documentclass[12pt,a4paper]{amsart}
\usepackage{amsmath,amssymb,amscd,mathrsfs}        
\usepackage[all]{xy}			           
\usepackage{epsfig,graphicx}			   
\usepackage{amsthm}
\usepackage{booktabs}
\usepackage{calc}							
\newcommand{\erase}[1]{}
\SelectTips{cm}{12}
\objectmargin={1pt}
\normalsize					   
\theoremstyle{definition}
\numberwithin{equation}{section}
\newtheorem{thm}{Theorem}[section]
\newtheorem{dfn}[thm]{Definition}

\newtheorem{prop}[thm]{Proposition}
\newtheorem{cor}[thm]{Corollary}
\newtheorem{lem}[thm]{Lemma}

\newtheorem{rem}[thm]{Remark}

\newcommand{\exam}{\noindent{\textbf{Example}:}}

\newcommand{\bb}[1]{{\mathbb{#1}}}
\newcommand{\mca}[1]{{\mathcal{#1}}}

\newcommand{\bracket}[1]{{\langle {#1} \rangle}}

\def\Aut{\mathop{\mathrm{Aut}}\nolimits}

\def\deg{\mathop{\mathrm{deg}}\nolimits}

\def\div{\mathop{\mathrm{div}}\nolimits}
\def\End{\mathop{\mathrm{End}}\nolimits}

\def\id{\mathop{\mathrm{id}}\nolimits}

\def\Pic{\mathop{\mathrm{Pic}}\nolimits}

\def\rank{\mathop{\mathrm{rank}}\nolimits}

\def\Sp{\mathop{\mathrm{Sp}}\nolimits}

\def\tr{\mathop{\mathrm{tr}}\nolimits}

\title{Hutchinson-Weber involutions degenerate exactly when the Jacobian is Comessatti}
\author{Hisanori Ohashi}\thanks{Research Institute for Mathematical Sciences, Kyoto University, Kyoto 606-8502, JAPAN}
\date{}

\begin{document}
\maketitle
\begin{abstract}
We consider the Jacobian Kummer surface $X$ of a genus two curve $C$. We prove that 
the Hutchinson-Weber involution on $X$ degenerates if and only if 
the Jacobian $J(C)$ is Comessatti. Also we give several conditions equivalent to this, 
which include the classical theorem of Humbert. 
The key notion is the Weber hexad. We include explanation of them and 
discuss the dependence between the conditions of main theorem for various Weber hexads.
It results in "the equivalence as dual six". 
We also give a detailed description of relevant moduli spaces.
As an application, we give a conceptual proof of the computation of the patching subgroup
for generic Hutchinson-Weber involutions.
\end{abstract}

\section{Introduction}

Let $J(C)$ be the Jacobian of a curve $C$ of genus two and $X$ the minimal 
desingularization of $\overline{X}=J(C)/\iota$, $\iota = -\id$. Here every variety 
we consider is over $\bb{C}$. 
$X=Km(J(C))$ is 
called a Jacobian Kummer surface which 
is well-known to be a $K3$ surface. 

In \cite{ohashi09} we classified fixed-point-free involutions on $X$, or equivalently
Enriques surfaces whose covering $K3$ surface is isomorphic to $X$,
under the condition that $X$ is Picard-general. They consist of 
$10$ switches, $15$ Hutchinson-G\"{o}pel involutions and $6$ Hutchinson-Weber 
involutions. In this paper we focus on the Hutchinson-Weber (HW) involutions;
the point of our discussion here 
is that we do not assume any kind of generality on the curve $C$.

HW involutions are closely related to the classical notion of {\em{Weber hexads}}
and associated {\em{Hessian models}} of $X$ as treated in \cite{dolgachev-keum}. 
We recall these notions in the first half of Section
\ref{hess}. Besides the definition itself, the equivalence relation "as dual six" 
plays an important role in this paper.
In the latter half, we study the singularities of Hessian models. 
We prove that the singularities of a Hessian model is either $10$ or $11$ nodes 
(Corollary \ref{11}). 
Moreover we show that $11$th node occurs exactly when the associated HW involution acquires 
fixed loci (Proposition \ref{41}), namely when the HW involution degenerates. 

On the other hand, 
an abelian surface $A$ is called a {\em{Comessatti surface}} if it has real multiplication 
in the maximal order $\mca{O}_{\bb{Q}(\sqrt{5})}$ of $\bb{Q} (\sqrt{5})$ \cite{hulek-lange}.
A classical theorem of Humbert characterizes Comessatti Jacobians in terms of the 
branch points $p_1,\cdots, p_6$ of the 
bicanonical map $C\rightarrow \text{(a conic)}\subset \bb{P}^2$, see for example \cite{vdg}.
\begin{thm}(Humbert)
The Jacobian $J(C)$ is Comessatti if and only if for a suitable labeling 
of branch points
there exists a conic $D$ 
which is inscribed to the pentagon $p_1\cdots p_5$ and passes through $p_6$.
\end{thm}
The projective dual of the six points $p_1, \cdots, p_6$ is the six branch lines
of the double plane model of Jacobian Kummer surface $X=Km(J(C))$.  
The dual of the conic $D$ induces a
new genus two curve on $J(C)$ different from the (translations of) theta divisors.
Equivalently, these curves are the pullbacks of the theta divisors by 
the automorphism $(\pm 1 + \sqrt{5})/2$.
We show that each of these curves passes through six $2$-torsion points, which form a Weber hexad
(Proposition \ref{weber}). As we expect easily, this curve corresponds exactly to the 
11th node of the Hessian model (Theorem \ref{main}, (1)$\Leftrightarrow$(4)).
Our main theorem is as follows.
\begin{thm}(Theorem \ref{main})
Let $C$ be a curve of genus two and $(X,W)$ its Jacobian Kummer surface and a Weber hexad on it.
Then the following conditions are equivalent.
\begin{enumerate}
\item The Hessian model $X_W$ acquires the $11$th node.
\item The Hutchinson-Weber involution $\sigma_W$ degenerates in the sense that 
it acquires fixed loci.
\item The unique twisted cubic $\overline{E}$ passing through the nodes 
$\{n_{\alpha}\}_{\alpha\in W}$ of $\overline{X}$
lies on the Kummer quartic surface $\overline{X}$. (Here
the strict transform $E\subset X$ satisfies the relations in Proposition \ref{11th}.)
\item The Jacobian $J(C)$ is a Comessatti surface and one curve $\Xi$ among (\ref{32})
passes through the $2$-torsion points corresponding to $W$.
\item In the double plane model (Proposition \ref{H-N}) projected from 
one node $n_{w_0}\ (w_0\in W)$, there exists an additional conic $E'\subset \bb{P}^2$ which passes
through the vertices of the pentagon formed by five images of $\{n_w; w\in W-\{w_0\}\}$ 
and tangent to the remaining branch line.
For example, when $W=\{0,12,23,34,45,51\}$ and $w_0=0$ as in Proposition \ref{H-N}, 
then the pentagon is formed by $l_1, \cdots, l_5$ and the last line is $l_6$. 
\end{enumerate}
\end{thm}
The equivalence between (4) and (5) is nothing but the above theorem of Humbert,
stated in the dual projective space.
But our theorem is a bit extended in the sense that we refer to the Weber hexads. 
Weber hexads are essentially divided into the "dual" six, Section \ref{hess},
and we can show that for equivalent Weber hexads, the conditions 
in the theorem are equivalent (Proposition \ref{six}).
Thus our theorem is more quantative than known even considered 
as the extention of theorem of Humbert, 
and the equivalence with conditions (1) and (2) are apparently new.
This theorem explains the title: "Hutchinson-Weber involutions degenerate 
exactly when the Jacobian is Comessatti". 

In Section \ref{mod} we give a detailed description of the moduli space
of Jacobian Kummer surfaces, Jacobian Kummer surfaces equipped with an equivalence class
of Weber hexad and the locus of degenerate Hutchinson-Weber involutions. We 
use the theory of period maps for $K3$ surfaces. We obtain the irreducibility of 
the moduli space of Comessatti Jacobian Kummer surfaces, Theorem \ref{moduli}.

In the last section, we give an application of this characterization 
to the computation of patching subgroups (see \cite{ohashi09}) 
of HW involutions. It seems interesting to the author 
that we can derive consequences to Picard-general
Jacobian Kummer surfaces by studying the degenerations. 

In this paper we restrict ourselves to genus two curves.
The suitable extention to the reducible principally polarized abelian surfaces,
namely the product of elliptic curves, is entirely left as a further problem.

\noindent {\bf{Acknowledgement.}} The author is grateful to Shigeru Mukai for 
fruitful discussions. His suggestion to Comessatti surfaces was the starting point
of this paper. Also the ingredient of the last section is fixed in the 
discussion with him. 

The author is supported by global COE program of Kyoto university.
This work was supported by KAKENHI 21840031.

\section{Jacobian Kummer surfaces}\label{JK}

Here we recall the construction of Jacobian Kummer surfaces and fix the notation.
We use the same indexing of divisors as in \cite{ohashi09}.

Let $C$ be a smooth projective curve of genus $2$. Let $J(C)=\Pic^0 (C)$ be its Jacobian
variety. It has the inversion morphism $\iota \colon x\mapsto -x$. We denote by $\overline{X}
=\overline{Km}(J(C))$ the quotient surface $J(C)/\iota$ and by $X=Km (J(C))$ the minimal
resolution. $X$ is a $K3$ surface associated to $C$ and called the 
{\em{Jacobian Kummer surface}} of $C$.
\[J(C) \quad \stackrel{/ \iota}{\longrightarrow} \quad \overline{X} \quad 
\stackrel{\text{min. res'n}}{\longleftarrow} \quad X\]

In the following we introduce several divisors on $X$ whose configuration is called 
the {\em{$(16)_6$-configuration}} on $X$. 
Recall that the morphism 
associated to the canonical system $|K_C|$ represents $C$ 
as a double cover of $\bb{P}^1$ ramified at $6$ Weierstrass points $p_1, \cdots, p_6\in C$.
Using them,
the set of $2$-torsion points of the Jacobian can be written as 
\begin{eqnarray*}
J(C)_2  &=&  \{ \alpha\in \Pic^0 (C)\mid 2\alpha \sim 0 \} \\
&=& \{0\} \cup \{ [p_i-p_j] \mid i\neq j\}.
\end{eqnarray*}
$2$-torsion points naturally correspond to the nodes $n_{\alpha}$ of $\overline{X}$ and 
exceptional curves $N_{\alpha}$ of $X$. 
On the other hand, the set of 
theta characteristics of $C$ can be written as 
\begin{eqnarray*}
S(C) &=& \{ \beta\in \Pic^1 (C) \mid 2\beta \sim K_C \} \\
 &=& \{ [p_i] \mid i=1,\cdots ,6 \} \cup \{ [p_i+p_j-p_k] \mid i\neq j\neq k\neq i\}.
\end{eqnarray*}
They also correspond to smooth rational curves on $\overline{X}$ and $X$ called {\em{tropes}};
the tropes $\overline{T}_{\beta} \subset \overline{X}$ and $T_{\beta}\subset X$ are
the strict transforms of the theta divisor 
\[ \Theta_{\beta}= \{ [p-\beta ] \in J(C) \mid p\in C\}. \] 
The incidence relation between $N_{\alpha}$ and $T_{\beta}$ is given by 
\begin{gather*}
(N_{\alpha},N_{\alpha '})=-2\delta_{\alpha,\alpha '},\qquad 
(T_{\beta},T_{\beta '})=-2\delta_{\beta,\beta '},\\
(N_{\alpha},T_{\beta})=1 \Leftrightarrow \alpha+\beta\in\{[p_1],[p_2],[p_3],[p_4],[p_5],[p_6]\},\\
(N_{\alpha},T_{\beta})=0 \qquad \text{otherwise}.
\end{gather*}
We will abbreviate $N_{[p_i-p_j]}$ to $N_{ij}$ and $T_{[p_i+p_j-p_k]}$ to $T_{ijk}$, etc.
We remark the relation $T_{ijk} = T_{lmn}$ for any permutation $i,\cdots, n$ of $1,\cdots,6$.

We will denote by $H$ the divisor class of $2T_1+N_0+\sum_{j=2}^6 N_{1j}$;
note that any analogous divisor
$2T_{\beta}+\sum_{(T_{\beta}, N_{\alpha})=1} N_{\alpha}$ gives the same divisor class as $H$.
The following fact is classically known and by this reasoning $\overline{X}$ is called 
the {\em{Kummer's quartic surface}}.
\begin{prop}(Kummer quartic model)\label{H}
The linear system $|H|$ induces an embedding of $\overline{X}$
into $\bb{P}^3$ as a quartic surface with sixteen nodes. 
The trope $\overline{T}_{\beta}\subset \overline{X}$ 
is a conic on $\overline{X}$ and the unique hyperplane containing $\overline{T}_{\beta}$
cuts $\overline{X}$ doubly along $\overline{T}_{\beta}$. 
\end{prop}
We usually regard $\overline{X}$ as embedded in $\bb{P}^3$. Projecting $\overline{X}$ 
from one of its nodes, say $n_{0}$, we obtain the following model.
\begin{prop}(double plane model)\label{H-N}
The linear system $|H-N_{0}|$ induces a generically two-to-one morphism of $X$ onto $\bb{P}^2$. 
It contracts the exceptional curves $N_{\alpha}$ other than $N_{0}$.  
If we denote the 
images of $T_{i}$ by $l_i$ for $i=1,\cdots ,6$, then $\overline{X}$
is a double cover of $\bb{P}^2$ branched 
along the union of six lines 
$\cup l_i$. The image of $N_0$ is a conic of which all $l_i$ are tangents. 
\end{prop}
We introduce two kinds of basic automorphisms. 
\begin{prop}\label{switch}
For each $\alpha_0\in J(C)_2$, the translation automorphism in $\alpha_0$ on $J(C)$
induces $X$ an automorphism
called a {\em{translation}}. It acts on $H^2 (X, \bb{Z})$ by: $H\mapsto H,\ N_{\alpha}\mapsto
N_{\alpha +\alpha_0}$ and $x\mapsto -x$ for $x$ orthogonal from $\{H, N_{\alpha}\}$. 

Similarly for each $\beta_0 \in S(C)$ there exists an automorphism of $X$ called a {\em{switch}}
that acts on $H^2 (X, \bb{Z})$ by: $H\mapsto 3H-\sum_{\alpha\in J(C)_2}N_{\alpha} ,
\ N_{\alpha}\mapsto
T_{\alpha +\beta_0}$ and $x\mapsto -x$ for $x$ orthogonal from $\{H, N_{\alpha}\}$.
\end{prop}
This proposition is valid for any Jacobian
Kummer surface $X$. Therefore we may say that translations and switches does not degenerate
under specialization of Jacobian Kummer surfaces.

\section{The Hessian model}\label{hess}

Let $X$ be a Jacobian Kummer surface associated to a curve $C$ of genus $2$.
In this section we focus on the {\em{Hessian model}} $X_W$ of $X$, treated for example in 
\cite{dolgachev-keum}.
After we give a self-contained proof of Proposition \ref{hessian},
we consider singularities of $X_W$. The point is that we do {\em{not}} assume that $C$ is 
general, in any sense.\\

\noindent {\textbf{Weber hexads.}}
The Hessian model $X_W$ is associated to a Weber hexad $W$. We first recall 
this notion. For the completeness sake, we include Lemma \ref{ww} which is 
already mentioned in \cite{dolgachev-keum} without a formal proof.

Let us define a symplectic form on $J(C)_2$
by $(\alpha, \alpha')= \# (\alpha\cap \alpha') \mod 2\in \bb{F}_2$, where we identify 
$\alpha$ with a two-element subset of $\{1,\cdots,6\}$. 
An affine $2$-dimensional subspace of $J(C)_2$ is called a {\em{G\"{o}pel tetrad}} 
if it is a translation of a totally isotropic $2$-dimensional 
linear subspace. Otherwise it is called a
{\em{Rosenhain tetrad}}; equivalently they are translations of nondegenerate $2$-dimensional 
linear subspaces.
We easily see that there are $60$ (resp. $80$) G\"{o}pel (resp.
Rosenhain) tetrads. 

A six-element subset of $J(C)_2$ is called a {\em{Weber hexad}} if it can be written in the form 
$G \ominus R$, where $G$ is a G\"{o}pel tetrad and $R$ is a Rosenhain tetrad such that $\# G\cap R
=1$.
\begin{lem}\label{ww} There are $192$ Weber hexads. 
Any Weber hexad has one of the forms of 
\[\{0,ij,jk,kl,lm,mi\} \text{ or } \{ij,jk,ki,il,jm,kn\},\]
where $\{i,\cdots, n\}$ is some permutation of $\{1,\cdots, 6\}$. 
\end{lem}
\begin{proof}
A permutation of letters $1,\cdots ,6$ induce an isometry of $J(C)_2$. This correspondence 
induces an isomorphism $\mca{S}_6\simeq \Sp (4,\bb{F}_2),$ hence
the affine isometry group of $J(C)_2$ can be written as
$(\bb{Z}/2\bb{Z})^4 \cdot \Sp (4,\bb{F}_2)\simeq 
J(C)_2 \cdot \mca{S}_6=:G$. 

First we show that $G$ acts on the set of Weber hexads transitively. 
Given $W$, we translate it appropriately and can assume it
is of the form $G\ominus R$ where $G\cap R=\{ 0 \}$. Then we easily 
check that the only possibility is $G=\{0,ij,kl,mn\}$ and $R=\{0,ik,km,mi\}$
for a suitable permutation $i,\cdots, n$ of $1,\cdots, 6$. This shows the transitivity. 

Next we compute the stabilizer subgroup $H$ of $W=\{ij,kl,mn,ik,km,mi\}$. 
The intersection $\mca{S}_6\cap H$ consists of 
six elements $\tau \sigma\tau\sigma^{-1}$ for $\tau\in \mca{S} (\{i,k,m\})$ and $\sigma =
(ij)(kl)(mn).$ On the other hand, for each $\alpha \in J(C)_2-W$ there exists a unique way of 
expressing $W$ as $G'\ominus R'$ with $G'\cap R'=\{\alpha\}$. Thus there exists six choices of
$\nu\in \mca{S}_6$ such that $\nu \alpha \in G$ sends $W$ onto itself.
In this way we obtain $6\cdot 10=60$ elements in $H$. Thus there are at most 
$2^4 6! /60=192$ Weber hexads.

Finally we easily see that the two standard 
forms in the statement gives at least $192$ Weber hexads.
Hence the lemma is proved. 
\end{proof}

Weber hexads are essentially one of the expressions of the "dual set" of $\{1,\cdots,6\}$. 
Recall that the symmetric group $\mca{S}_6$ has two permutation representations. 
One is the natural representation on $\{1,\cdots,6\}$ and the other is 
the one twisted by the outer automorphism. 

In \cite{ohashi09} we proved that if the curve $C$ is generic, then 
the $192$ Hutchinson-Weber involutions $\sigma_W$
(Section \ref{HWi}) are divided into exactly 
six conjugacy classes in $\Aut (X)$.
We can see that the permutation on the labels 
of Weierstrass points of $C$ and the permutation on these six conjugacy classes are
related by an outer automorphism, hence these six conjugacy classes can be regarded as 
the dual set.

The conjugacy relation between Hutchinson-Weber involutions are given by translations 
and switches of Proposition \ref{switch} and corresponds to
the following equivalence relation between Weber hexads:
it is 
generated by $W\sim W+\alpha$ ($\alpha\in J(C)_2$) and $W=G\ominus R\sim G\ominus R^{\perp}$
(when $G\cap R=\{ 0\}$). We refer this equivalence relation as the equivalence as dual six.
In this paper we will consider the degenerate cases of Hutchinson-Weber involutions and 
clarify the meaning of this equivalence relation.

In Remark (2) after Proposition 7.4 of \cite{ohashi09} we have given
one possible explicit description of the equivalence as dual six.
Here let us give more visible one.

Let us recall the classical description of the dual set, found for example in \cite{atlas}.
An element in $\mca{S}_6$ of the form $(ij)$ is called a {\em{duad}}; similarly 
$(ij)(kl)(mn)$ is called a {\em{syntheme}}; a five-element set is called a {\em{total}}
if it consists of five synthemes that contain all fifteen duads.
There are exactly six totals and this is the classical description of the 
dual set.

As in the lemma, a Weber hexad is one of the two types. The picture below indicates the 
correspondence from a Weber hexad to a total. 
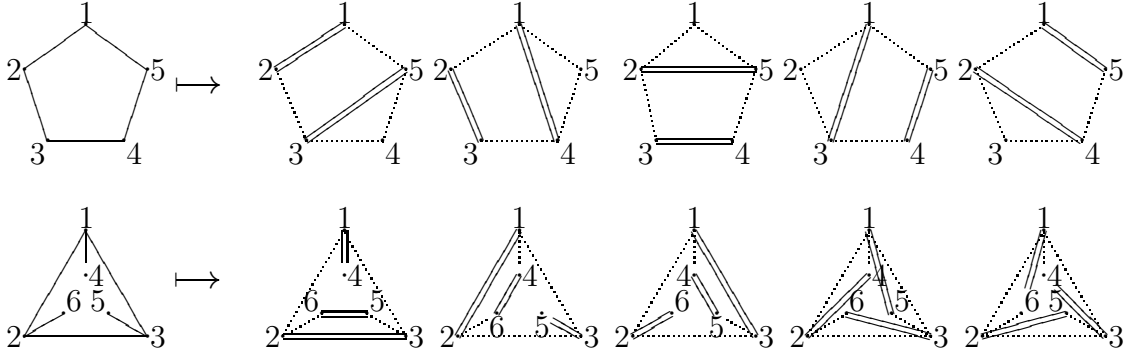
\[\begin{xy}
(0,0) *{\cdot}*+!UR{3}="A",
(10,0) *{\cdot}*+!UL{4}="B",
(13.090,9.511) *{\cdot}*+!L{5}="C",
(5,15.3885) *{\cdot}*+!D{1}="D",
(-3.090,9.511) *{\cdot}*+!R{2}="E",
(34,0) *{\cdot}*+!UR{3}="Ag",
(44,0) *{\cdot}*+!UL{4}="Bg",
(47.090,9.511) *{\cdot}*+!L{5}="Cg",
(39,15.3885) *{\cdot}*+!D{1}="Dg",
(30.010,9.511) *{\cdot}*+!R{2}="Eg",
(57,0) *{\cdot}*+!UR{3}="Ah",
(67,0) *{\cdot}*+!UL{4}="Bh",
(70.090,9.511) *{\cdot}*+!L{5}="Ch",
(62,15.3885) *{\cdot}*+!D{1}="Dh",
(53.010,9.511) *{\cdot}*+!R{2}="Eh",
(80,0) *{\cdot}*+!UR{3}="Ai",
(90,0) *{\cdot}*+!UL{4}="Bi",
(93.090,9.511) *{\cdot}*+!L{5}="Ci",
(85,15.3885) *{\cdot}*+!D{1}="Di",
(78.010,9.511) *{\cdot}*+!R{2}="Ei",
(103,0) *{\cdot}*+!UR{3}="Aj",
(113,0) *{\cdot}*+!UL{4}="Bj",
(116.090,9.511) *{\cdot}*+!L{5}="Cj",
(108,15.3885) *{\cdot}*+!D{1}="Dj",
(99.010,9.511) *{\cdot}*+!R{2}="Ej",
(126,0) *{\cdot}*+!UR{3}="Ak",
(136,0) *{\cdot}*+!UL{4}="Bk",
(139.090,9.511) *{\cdot}*+!L{5}="Ck",
(131,15.3885) *{\cdot}*+!D{1}="Dk",
(122.010,9.511) *{\cdot}*+!R{2}="Ek",
\ar@{-}"A";"B"
\ar@{-}"B";"C"
\ar@{-}"C";"D"
\ar@{-}"D";"E"
\ar@{-}"E";"A"
\ar@{|->}(17,7.5);(22.5,7.5)
\ar@{.}"Ag";"Bg"
\ar@{.}"Bg";"Cg"
\ar@{.}"Cg";"Dg"
\ar@2{-}"Dg";"Eg"
\ar@{.}"Eg";"Ag"
\ar@2{-}"Ag";"Cg"
\ar@{.}"Ah";"Bh"
\ar@{.}"Bh";"Ch"
\ar@{.}"Ch";"Dh"
\ar@{.}"Dh";"Eh"
\ar@2{-}"Eh";"Ah"
\ar@2{-}"Bh";"Dh"
\ar@2{-}"Ai";"Bi"
\ar@{.}"Bi";"Ci"
\ar@{.}"Ci";"Di"
\ar@{.}"Di";"Ei"
\ar@{.}"Ei";"Ai"
\ar@2{-}"Ci";"Ei"
\ar@{.}"Aj";"Bj"
\ar@2{-}"Bj";"Cj"
\ar@{.}"Cj";"Dj"
\ar@{.}"Dj";"Ej"
\ar@{.}"Ej";"Aj"
\ar@2{-}"Aj";"Dj"
\ar@{.}"Ak";"Bk"
\ar@{.}"Bk";"Ck"
\ar@2{-}"Ck";"Dk"
\ar@{.}"Dk";"Ek"
\ar@{.}"Ek";"Ak"
\ar@2{-}"Bk";"Ek"
\end{xy}\]
\[\begin{xy}
(0,0) *{\cdot}*+!R{2}="AA",
(5.129,2.961) *{\cdot}*+!DL{6}="BB",
(8.09,8.09) *{\cdot}*+!L{4}="CC",
(8.09,14.012) *{\cdot}*+!D{1}="DD",
(11.051,2.961) *{\cdot}*+!DR{5}="EE",
(16.18,0) *{\cdot}*+!L{3}="FF",
(34,0) *{\cdot}*+!R{2}="AAg",
(39.129,2.961) *{\cdot}*+!DR{6}="BBg",
(42.09,8.09) *{\cdot}*+!L{4}="CCg",
(42.09,14.012) *{\cdot}*+!D{1}="DDg",
(45.051,2.961) *{\cdot}*+!DL{5}="EEg",
(50.18,0) *{\cdot}*+!L{3}="FFg",
(57,0) *{\cdot}*+!R{2}="AAh",
(62.129,2.961) *{\cdot}*+!U{6}="BBh",
(65.09,8.09) *{\cdot}*+!L{4}="CCh",
(65.09,14.012) *{\cdot}*+!D{1}="DDh",
(68.051,2.961) *{\cdot}*+!U{5}="EEh",
(73.18,0) *{\cdot}*+!L{3}="FFh",
(80,0) *{\cdot}*+!R{2}="AAi",
(85.129,2.961) *{\cdot}*+!DL{6}="BBi",
(88.09,8.09) *{\cdot}*+!R{4}="CCi",
(88.09,14.012) *{\cdot}*+!D{1}="DDi",
(91.051,2.961) *{\cdot}*+!U{5}="EEi",
(96.18,0) *{\cdot}*+!L{3}="FFi",
(103,0) *{\cdot}*+!R{2}="AAj",
(108.129,2.961) *{\cdot}*+!DL{6}="BBj",
(111.09,8.09) *{\cdot}*+!L{4}="CCj",
(111.09,14.012) *{\cdot}*+!D{1}="DDj",
(114.051,2.961) *{\cdot}*+!DL{5}="EEj",
(119.18,0) *{\cdot}*+!L{3}="FFj",
(126,0) *{\cdot}*+!R{2}="AAk",
(131.129,2.961) *{\cdot}*+!DL{6}="BBk",
(134.09,8.09) *{\cdot}*+!L{4}="CCk",
(134.09,14.012) *{\cdot}*+!D{1}="DDk",
(137.051,2.961) *{\cdot}*+!DR{5}="EEk",
(142.18,0) *{\cdot}*+!L{3}="FFk",
\ar@{-}"AA";"DD"
\ar@{-}"DD";"FF"
\ar@{-}"FF";"AA"
\ar@{-}"AA";"BB"
\ar@{-}"CC";"DD"
\ar@{-}"EE";"FF"
\ar@{|->}(20,7.5);(25.5,7.5)
\ar@{.}"AAg";"DDg"
\ar@{.}"DDg";"FFg"
\ar@2{-}"FFg";"AAg"
\ar@{.}"AAg";"BBg"
\ar@2{-}"CCg";"DDg"
\ar@{.}"EEg";"FFg"
\ar@2{-}"BBg";"EEg"
\ar@2{-}"AAh";"DDh"
\ar@{.}"DDh";"FFh"
\ar@{.}"FFh";"AAh"
\ar@{.}"AAh";"BBh"
\ar@{.}"CCh";"DDh"
\ar@2{-}"EEh";"FFh"
\ar@2{-}"BBh";"CCh"
\ar@{.}"AAi";"DDi"
\ar@2{-}"DDi";"FFi"
\ar@{.}"FFi";"AAi"
\ar@2{-}"AAi";"BBi"
\ar@{.}"CCi";"DDi"
\ar@{.}"EEi";"FFi"
\ar@2{-}"CCi";"EEi"
\ar@{.}"AAj";"DDj"
\ar@{.}"DDj";"FFj"
\ar@{.}"FFj";"AAj"
\ar@{.}"AAj";"BBj"
\ar@{.}"CCj";"DDj"
\ar@{.}"EEj";"FFj"
\ar@2{-}"AAj";"CCj"
\ar@2{-}"BBj";"FFj"
\ar@2{-}"DDj";"EEj"
\ar@{.}"AAk";"DDk"
\ar@{.}"DDk";"FFk"
\ar@{.}"FFk";"AAk"
\ar@{.}"AAk";"BBk"
\ar@{.}"CCk";"DDk"
\ar@{.}"EEk";"FFk"
\ar@2{-}"FFk";"CCk"
\ar@2{-}"BBk";"DDk"
\ar@2{-}"AAk";"EEk"
\end{xy}\]
The former picture indicates the correspondence
\begin{equation*}
\begin{split}
W= \{ &0,12,23,34,45,51\} \mapsto  \\
&\{(12)(35)(46),\ (14)(23)(56),\ (16)(25)(34),\ (13)(26)(45),\ (15)(24)(36)\},
\end{split}
\end{equation*}
where the letter $6$ is regarded as distinguished.
The latter one indicates
\begin{equation*}
\begin{split}
W= \{ &12,23,31,14,26,35\} \mapsto \\
&\{(14)(23)(56),\ (12)(35)(46),\ (13)(26)(45),\ (15)(24)(36),\ (16)(25)(34)\}.
\end{split}
\end{equation*}
We obtain the same total.
Thus we see that these $W$ correspond to the same dual element.\\

\noindent {\textbf{The Hessian model.}}
The Hessian model $X_W$ is constructed for every Weber hexad $W$. 
Hence in the following, we consider the pair $(X,W)$
consisting of a Jacobian Kummer surface $X$ and a Weber hexad $W$.
The next proposition is known to experts, 
see \cite{dolgachev-keum} and its references, but our algebraic proof is more suited 
for what follows.
\begin{prop}\label{hessian}
(The Hessian model) The linear system $|L|:=|2H-\sum_{\alpha\in W} N_{\alpha}|$ 
maps $X$ birationally to a quartic surface $X_W$ whose equation is of the form
\[s_1+\cdots+s_5=0,\ s_1s_2s_3s_4s_5(\lambda_1/s_1+\cdots+\lambda_5/s_5)=0,\]
where $\lambda_i$ are nonzero constants and $s_i$ are homogeneous coordinates of $\bb{P}^4$.
\end{prop}
\begin{proof}
As indicated above, Weber hexads are unique up to the affine symplectic group.
The group $(\bb{Z}/2\bb{Z})^4$ lifts to translation automorphisms of $J(C)$
in the elements of $J(C)_2$,
which commute with the quotient by $\iota$. 
The group $\Sp (4,\bb{F}_2)\simeq \mathfrak{S}_6$ acts as 
permutations of the letters. So it is enough to see the proposition for a particular 
Weber hexad. Let us take $W=\{12,23,31,14,25,36\}$.

Let us consider the divisors (cf. \cite{dolgachev-keum})
\begin{eqnarray*}
S_1&=&T_2+T_3+T_{124}+T_{134}+N_0+N_{24}+N_{26}+N_{34}+N_{35}+N_{56},\\
S_2&=&T_{123}+T_{145}+T_{134}+T_{125}+N_{15}+N_{26}+N_{34}+N_{45}+N_{46}+N_{56},\\
S_3&=&T_{1}+T_{3}+T_{125}+T_{146}+N_{0}+N_{15}+N_{16}+N_{34}+N_{35}+N_{46},\\
S_4&=&T_{123}+T_{124}+T_{146}+T_{136}+N_{16}+N_{24}+N_{35}+N_{45}+N_{46}+N_{56},\\
S_5&=&T_{1}+T_{2}+T_{136}+T_{145}+N_{0}+N_{15}+N_{16}+N_{24}+N_{26}+N_{45}.
\end{eqnarray*}
It is easy to see that these divisors belong to $|L|$ and a careful
check using them 
shows that $|L|$ is base-point-free. Thus the associated map $\varphi=\varphi_{L}$ is a morphism.
By the Kawamata-Viehweg vanishing and Riemann-Roch
we see that $h^0 (L)=4.$ Hence the sections $s_i\in H^0(L)$ corresponding to $S_i$ are linearly
dependent. On the other hand, 
by evaluating at general points of $N_{\alpha}$ for several $\alpha$, 
we can check that every four among $\{s_1,\cdots, s_5\}$ is linearly independent.
This shows that, up to adjusting the scalars, we can assume $\sum_{i=0}^5 s_i=0$. 
By this equation, we regard the morphism $\varphi$ as the morphism into 
$\{\sum_{i=1}^5 s_i=0\}\simeq \bb{P}^3$ in $\bb{P}^4$. 
We denote by $X_W$ the image of $\varphi$.

Let us denote the hyperplane $\{s_i=0\}$ by $H_i$.
Ten divisors $T_{\beta}$ appearing in $\cup S_i$ are mapped to a line on $X_W$.
They appear with multiplicity two in $\cup S_i$, hence if $T_{\beta}\subset S_i\cap S_j$
then we can write $\varphi (T_{\beta})=H_i\cap H_j=:L_{ij}$. 
Similarly, the ten divisors $N_{\alpha}$ appearing in $\cup S_i$ are 
contracted to a point on $\overline{X}_W$.
They appear exactly three times in $\cup S_i$, so 
we can write $\varphi (N_{\alpha})=H_i\cap H_j\cap H_k=:P_{ijk}$ if 
$N_{\alpha}\subset S_i \cap S_j \cap S_k$.

Let us look at hyperplane section $H_1\cap X_W$ closely.
It contains four lines $L_{1j},\ j=2,\cdots, 5$, namely the images of $T_{134}, T_3,
T_{124}$ and $T_2$. 
General points of these four tropes are separated each other by divisors $S_i$. Thus 
the hyperplane $H_1$ cuts $X_W$ along four distinct lines. This implies
that $\deg X_W=4$ and $\varphi$ is birational. 
Let $f(s_1,s_2,s_3,s_4)$ be the quartic equation of $X_W$, $s_5$ being 
substituted by $-(s_1+\cdots+s_4)$. 
The argument above shows that $f (0,s_2,
s_3,s_4)$ is a multiple of $s_2,s_3,s_4,-(s_2+s_3+s_4)$. Similar consequences hold 
for $s_2=0,s_3=0,s_4=0$.
In summary it follows that $f$ is a linear combination of the terms
\begin{eqnarray*}
s_1s_2s_3s_4,s_2s_3s_4(s_2+s_3+s_4),s_1s_3s_4(s_1+s_3+s_4),\\
s_1s_2s_4(s_1+s_2+s_4),s_1s_2s_3(s_1+s_2+s_3).
\end{eqnarray*}
Using $s_5$, these terms can be written by a linear combination of 
\[s_1s_2s_3s_4s_5/s_i,\ i=1,\cdots,5.\]
Thus we derived the equation. $\lambda_i\neq 0$ is because $X_W$ is irreducible.
\end{proof}
We can derive several consequences from this proposition. 
\begin{cor}\label{nml} $X_W$ is normal. \end{cor}
\begin{proof}
Let $\psi \colon X\rightarrow Y$ be the morphism 
which contracts all the $(-2)$-curves on $X$ orthogonal to $L$.
$Y$ is a normal surface with at most rational double points,
and the canonical sheaf of $Y$ is trivial.
Since the exceptional sets of $\psi$ and $\varphi$ coincide,
$\varphi$ factors as $\varphi = \nu \psi$. 
By the adjunction formula $K_{X_W}$ is also trivial, so $\nu$ is 
etale in codimension one, hence $X_W$ is regular in codimension one.
Since $X_W$ is a complete intersection, by Serre's criterion we see 
that $X_W$ is normal.
\end{proof}
\begin{cor}
Each $P_{ijk}$ is an ordinary node. 
\end{cor}
\begin{proof}
This follows from $\varphi^{-1} (P_{ijk})=S_i\cap S_j\cap S_k= N_{\alpha}$.
\end{proof}
Here we put an observation. 
By a direct checking we see $N_{\alpha}\subset \cup S_i$ 
if and only if $\alpha\in J(C)_2-W$. Thus 
\begin{equation}\label{10nodes}
\text{For $\alpha\in J(C)_2-W$, $N_{\alpha}$ is contracted to an ordinary node on $X_W$.}
\end{equation}
\begin{prop}\label{11th}
Suppose a $(-2)$-curve $E$ different from $\{N_{\alpha}\}$
is contracted by $\varphi$. Then $E$ has to satisfy the relations
\begin{equation*}
\begin{cases}
(E,N_{\alpha})=0 & \alpha\in J(C)_2-W,\\
(E,N_{\alpha})=1 & \alpha\in W,\\
(E,H)=3. & 
\end{cases}\end{equation*}
Moreover, such $E$ is unique if exists.
\end{prop}
\begin{proof}
By the previous corollary $E$ and exceptional $N_{\alpha}$ does not meet, otherwise 
the singularity is not a node. Hence 
$(E,N_{\alpha})=0$ for $\alpha\in J(C)_2-W$. 
Let us consider $N_{\alpha}$ for $\alpha\in W$.
By the projection formula $(\varphi_* (N_{\alpha}),\mca{O}_{X_W} (1))
=(N_{\alpha},L)=2$ hence we see that $\varphi_* (N_{\alpha})$ is a cycle of degree $2$. 
It is irreducible and reduced by Zariski main theorem, so
$\varphi_*(N_{\alpha})=\varphi (N_{\alpha})$ is a smooth conic. 
Hence $\varphi$ induces the isomorphism $N_{\alpha} \stackrel{\sim}{\rightarrow}
\varphi (N_{\alpha})$. 

If $N_{\alpha},\ \alpha\in W$ intersects 
the exceptional $E$ with intersection number $\ge 2$, then clearly 
$\varphi (N_{\alpha})$ acquires a singular point, a contradiction.
See the picture below.
It follows $(E, N_{\alpha})=0$ or $1$. On the other hand we have
$(E,L)=(E,2H-\sum_{\alpha\in W} N_{\alpha})=0$, thus $0\le (E,H)\le 3$. 
$(E,H)=0$ is prohibited by Proposition \ref{H}.

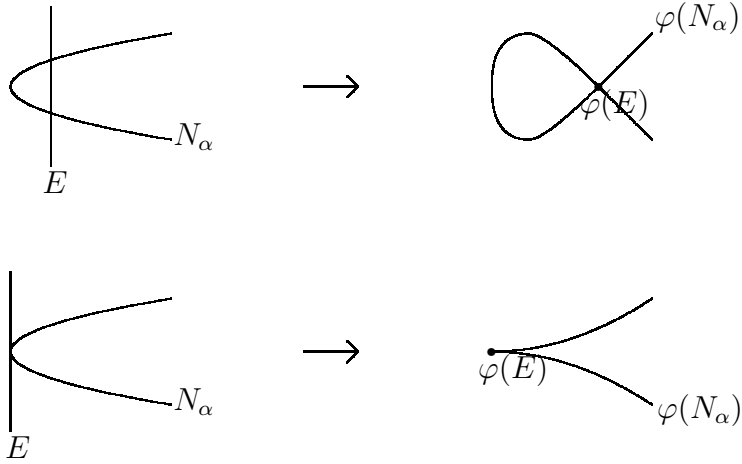
\begin{picture}(400,200)
\thinlines
\qbezier(75,120)(75,150)(75,180)
\put(72,110){$E$}
\qbezier(120,130)(0,150)(120,170)
\put(121,128){$N_{\alpha}$}
\qbezier(170,150)(170,150)(190,150)
\qbezier(186,154)(186,154)(190,150)
\qbezier(186,146)(186,146)(190,150)
\put(280,150){\circle*{3}}
\put(273,140){$\varphi (E)$}
\put(301,173){$\varphi (N_{\alpha})$}
\qbezier(300,170)(300,170)(280,150)
\qbezier(280,150)(260,130)(254,130)
\qbezier(254,130)(240,130)(240,150)
\qbezier(240,150)(240,170)(254,170)
\qbezier(254,170)(260,170)(280,150)
\qbezier(280,150)(280,150)(300,130)
\qbezier(60,20)(60,50)(60,80)
\put(58,10){$E$}
\qbezier(120,30)(0,50)(120,70)
\put(121,28){$N_{\alpha}$}
\qbezier(170,50)(170,50)(190,50)
\qbezier(186,54)(186,54)(190,50)
\qbezier(186,46)(186,46)(190,50)
\put(240,50){\circle*{3}}
\qbezier(240,50)(270,50)(300,70)
\qbezier(240,50)(270,50)(300,30)
\put(235,40){$\varphi (E)$}
\put(302,25){$\varphi (N_{\alpha})$}
\end{picture}

Let us denote by $\overline{E}$ the corresponding curve on $\overline{X}=J(C)/\iota$.
This is a smooth rational curve passing through $2(H,E)$ nodes.

Assume $(H,E)=1$.
Then the inverse image of $\overline{E}$ in $J(C)$ is 
a double cover branched at two points of $\overline{E}$, hence
a rational curve. 
Since an abelian 
surface doesn't contain any rational curve, a contradiction.

Assume $(H,E)=2$. Then $\overline{E}$ is an irreducible conic in $\bb{P}^3$ 
passing through four nodes belonging to $W$. These nodes therefore must be 
contained in a hyperplane
of $\bb{P}^3$, which contradicts to lemma below.

Assume $(H,E)=3$. Then $\overline{E}$ is a 
cubic curve passing through six nodes of $W$. By the lemma below, it is 
exactly the twisted cubic 
determined by $W$ and the uniqueness follows from 
the Steiner construction \cite{GH}.
Thus the whole proposition is reduced to the next lemma. 
\end{proof}
\begin{lem}
If we identify $W$ with the corresponding nodes $n_{\alpha}$ of $\overline{X}$, then
no four points of Weber hexad $W$ is coplanar. Namely they are in general position as to 
$\mca{O} (1)$.
\end{lem}
\begin{proof}
We begin by showing that no three nodes of $\overline{X}$ are collinear.
Assume the contrary. Then since $\overline{X}$ is a quartic surface, the line $l$ containing
them lies on $\overline{X}$ and $(l,H)=1$ (intersection numbers are computed on $X$,
so we identify $l$ with its strict transform on $X$). 
By the relation 
\begin{equation}\label{aaaa}
H\sim 2T_{\beta}+\sum_{(N_{\alpha},T_{\beta})=1} N_{\alpha}
\end{equation}
we see 
that $(l,T_{\beta})=0$. On the other hand, clearly for (at least) three $\alpha$ we 
have $(l,N_{\alpha})=1$. Summing up the relation (\ref{aaaa}) over $\beta\in S(C)$, we obtain
$16H \sim 2 \sum T_{\beta} + 6 \sum N_{\alpha}$. The left-hand-side intersects $l$ with $16$
but the right-hand-side intersects $l$ with at least $6\cdot 3=18$, hence we obtain
a contradiction.

Next, because the incidence relation between nodes and tropes 
is preserved under the affine symplectic group $G$, it suffices to prove the lemma 
in case $W=\{12,23,31,14,25,36\}$ for example.
Choose four points $\{12,23,14,25\}$. We see that the trope
$T_{2}$ passes through the points $n_{12},n_{23},n_{25}$ and doesn't through $n_{14}$. 
By Proposition \ref{H} a trope is a conic and coincides with the hyperplane section. 
Thus the four points are not coplanar. Similarly for every four points from $W$,
we can find a trope containing three 
but not the remaining fourth point. Thus we obtain the lemma.
\end{proof}
\begin{cor}\label{11}
The singularities of $X_W$ consist of $10$ or $11$ ordinary nodes.
If $X$ is Picard-general, i.e., the Picard number of $X$ is $17$, then 
$X_W$ has only $10$ nodes.
\end{cor}
\begin{proof}
The former part follows from the previous 
proposition. For the latter, we recall that for Picard-general
$X$, $NS(X)$ is generated over $\bb{Q}$ by the divisors $\{H,N_{\alpha}\}$. 
There exist no elements satisfying the condition for $E$ above, so it doesn't 
exist.
\end{proof}

\section{Hutchinson-Weber involutions and Comessatti surfaces}\label{HWi}

\noindent {\textbf{Hutchinson-Weber involutions.}}
We keep the assumption that $X$ is a Jacobian Kummer surface associated to $C$.
Let us consider the Hessian model $X_W: \{\sum s_i=\sum \lambda_i/s_i =0\}$ defined 
in $\bb{P}^4$.
We consider the Hutchinson-Weber involution defined 
by $\sigma_W \colon (s_1,\cdots, s_5)\mapsto (\lambda_1/s_1,\cdots,\lambda_5/s_5)$. It induces a
biregular involution on $X$, which also we denote by $\sigma_W$.

\begin{prop}\label{41}
The following are equivalent.
\begin{enumerate}
\item There exists one more node other than $10$ nodes of (\ref{10nodes}).
\item $\sigma_W$ is not fixed-point-free.
\item For some choice of signatures, we have 
$\pm \sqrt{\lambda_1}\pm \cdots \pm \sqrt{\lambda_5}= 0$.
\end{enumerate}
\end{prop}
\begin{proof} 
(1) $\Leftrightarrow$ (3): The 11th node $p$ corresponds to the rational curve $E$ of Proposition
\ref{11th}. The hyperplane $\{s_i=0\}$ cuts $X_W$ along four lines in $\bb{P}^2$ 
in general position.
Its singularities are $6$ nodes appearing from (\ref{10nodes}). Hence $p$ 
is located inside
the open set $\{s_1\cdots s_5\neq 0\}$.
By the Jacobian criterion of smoothness, 
we easily deduce that the 11th node should satisfy the relation 
\begin{equation}\label{rel}
\rank
\begin{pmatrix}
1 & 1 & 1 & 1 & 1 \\
\frac{\lambda_1}{s_1^2} & \frac{\lambda_2}{s_2^2} & \frac{\lambda_3}{s_3^2} & \frac{\lambda_4}{s_4^2} & \frac{\lambda_5}{s_5^2}
\end{pmatrix}
\leq 1.
\end{equation}
Thus its existence is equivalent to the condition (3).

(2) $\Leftrightarrow$ (3): 
First we notice that $\sigma_W$ sends the line 
$L_{ij}=\{s_i=s_j=0\}$ to the point $P_{klm}=\{s_k=s_l=s_m=0\}$,
where $\{i,\cdots, m\}$ is an arbitrary permutation of $\{1,\cdots,5\}$. Vice versa, $P_{klm}$ is 
sent to $L_{ij}$ since $\sigma_W$ is an involution. Thus a fixed point can occur only 
inside the open set $\{s_1\cdots s_5\neq 0\}$. 
Here clearly the fixed point is given by the further condition
\[\frac{\lambda_1}{s_1^2} = \frac{\lambda_2}{s_2^2} = \cdots = \frac{\lambda_5}{s_5^2},\]
which is equivalent to the relation
(\ref{rel}). Thus it is equivalent to the condition (3)
\end{proof}
By the above proof, the fixed point of $\sigma_W$ corresponds to the eleventh node of $X_W$. 
In this case since $\sigma_W$ is non-symplectic, it 
fixes the whole exceptional curve $E$.
\begin{rem} 
The equation 
\[s_1+\cdots+s_5= \frac{s_1^3}{\lambda_1}+\cdots+\frac{s_5^3}{\lambda_5}=0,\]
which defines a cubic surface, is called the Sylvester form of the cubic. 
It is known that generic cubic surface can be written in the Sylvester form 
in a unique way up to permutations and homothethy, so this equation is well-studied
in connection with the moduli problem of cubic surfaces.
Our $X_W$ is exactly of the form of "Hessian surface" of this cubic, hence the name.
We note that there are four parameters for cubic surfaces, while there are three
parameters for Jacobian Kummer surfaces. Hence general Hessian $K3$ surfaces can not 
be obtained as the Hessian model of Jacobian Kummer surfaces.

It is known that the condition (3) in the preceding proposition represents 
the locus of singular cubic surfaces, see for example \cite{dardanelli-vg}.
Genus two curves and singular cubics constitute the Kummer divisor and the boundary divisor
inside the four-dimensional moduli space of cubic surfaces, respectively.
Thus our object, the degenerations of Hutchinson-Weber involutions,
correspond to the intersection of these divisors.
\end{rem}

\noindent {\textbf{Comessatti surfaces.}}
We begin by the definition.
\begin{dfn}\label{come}
An abelian surface $A$ is called a Comessatti surface if it has real multiplication 
in the maximal order $\mca{O}_{\bb{Q}(\sqrt{5})}$ of $\bb{Q} (\sqrt{5})$, 
i.e., if $\mca{O}_{\bb{Q}(\sqrt{5})}=\bb{Z} [(1+\sqrt{5})/2]\subset \End (A)$.  
\end{dfn}
Let us suppose that the Jacobian $J(C)=:A$ is at the same time Comessatti.
We fix a theta divisor
$\Theta =\Theta_{\beta}$ and let
$\varphi \mapsto \varphi '$ be the Rosati involution 
on $\End (A)$
associated to $\mca{O}_A (\Theta )$.
We note that by the positivity of Rosati involution, 
it acts on $\mca{O}_{\bb{Q}(\sqrt{5})}$ trivially.
By definition we can
consider the endomorphism $\varepsilon = (1+\sqrt{5})/2$
which is in fact an automorphism.
By \cite[Section 21]{mumford}, we get 
\begin{equation}\label{3}
 (\Theta, \varepsilon^* \Theta )= \tr_{\bb{Q}(\sqrt{5})/\bb{Q}} (\varepsilon \varepsilon')= 3.
\end{equation}
Since $\Theta_{\beta}$ contains six $2$-torsion points $[\beta -p_i]\ (i=1,\cdots ,6)$, 
$\Xi := \varepsilon^* \Theta$ also contains 
six $2$-torsion points $w_i=\varepsilon^{-1} ([\beta -p_i])$.
\begin{prop}\label{weber}
$W=\{ w_1, \cdots, w_6\}$ is a Weber hexad.
\end{prop}
\begin{proof}
Clearly the sum $\sum w_i$ is zero. 
Hence the partial sums $w_1+w_2+w_3$ and $w_4+w_5+w_6$ are equal. We put this element as $x$. 
It is easy to see $x\not\in W$.
Then $I=\{x,w_1,w_2,w_3\}$ and $J=\{x,w_4,w_5,w_6\}$ are affine $2$-dimensional subspaces
with $I\ominus J=W$. 
Any $2$-dimensional affine subspace is either a Rosenhain tetrad or a G\"{o}pel tetrad.
Therefore as to the types three possibilities occur. 
Up to translation we can assume $x=0$ without loss of generality. 

Assume that $I,J$ are both Rosenhain tetrads. We can put $I=\{0,12,23,31\}$ by permutation. Then 
$J$ can be either $\{0,14,45,51\}$ or $\{0,45,56,64\}$ up to permutation. 
In both cases we deduce that $(\Xi, \Theta_{123})\ge 4$ and get contradiction to
(\ref{3}).

Assume that $I,J$ are both G\"{o}pel tetrads. We can put $I=\{0,12,34,56\}$ by permutation. 
Then $J$ can be only $\{0,23,45,61\}$ up to permutation. Again $(\Xi, \Theta_{123})\ge 4$ and
contradiction. 

Thus we have $W=I\ominus J$ with $I,J$ are Rosenhain and G\"{o}pel. Hence $W$ is a 
Weber hexad.
\end{proof}
We easily observe that the equation (\ref{3}) and the above proposition is also true 
for $\eta = \varepsilon^{-1} = (-1+\sqrt{5})/2$ instead of $\varepsilon$. 
We thus obtain the following set of genus two curves on $J(C)$.
\begin{equation}\label{32}
\mca{W} = \{\varepsilon^* \Theta_{\beta}\mid \beta\in S(C)\} \cup
\{\eta^*\Theta_{\beta}\mid \beta\in S(C)\}.
\end{equation}
Here for the convenience we note that under the isomorphism 
\begin{equation}\label{symm}
NS(J(C)) \stackrel{\sim}{\rightarrow} \End^{\text{sym}} (J(C)) = \{\varphi\in \End (J(C))\mid 
\varphi' = \varphi\},
\end{equation}
\cite[Chapter 5]{birkenhake-lange},
we have $c_1 (\mca{O} (\Theta_{\beta})) \mapsto \id$ and $c_1 (\mca{O} 
(\varepsilon^* \Theta_{\beta})) \mapsto \varepsilon^2$. By the relation $\varepsilon^4 -3\varepsilon^2 +1=0$, 
we obtain the algebraic equivalence $\eta^* \Theta_{\beta} \approx
3\Theta_{\beta}-\varepsilon^* \Theta_{\beta}.$

Recall that $\iota =-\id$ is the inversion. 
\begin{lem}
Let $F$ be a smooth genus two curve on $J(C)$. Then $\iota^*F=F \text{ (as a set)}$ if and only if 
$F$ passes through six $2$-torsion points.
\end{lem}
\begin{proof}
Let us assume $\iota^* F=F$. Then,
since we can regard $J(C)=\Pic^0 (F)$, $\iota \mid_{F}$ acts as a hyperelliptic involution
and it has $6$ fixed points.
Conversely suppose $F$ contains six $2$-torsion points. $\iota$ acts on $H^2 (J(C), \bb{Z})$ 
trivially, hence $(F, \iota^*F)=(F^2)=2$ and $\# F \cap \iota^*F \geq 6$ imply $F=\iota^* F$. 
\end{proof}
\begin{lem}\label{cha}
Curves $F\in \mca{W}$ are characterized by the conditions 
\[\iota^* F=F \text{ (as a set) and } F\approx \varepsilon^* \Theta \text{ or }
\eta^*\Theta, \]
where $\approx$ is the algebraic equivalence. 
Moreover every $F\in \mca{W}$ passes through distinct Weber hexads each other. Hence we 
obtain $32$ Weber hexads from $\mca{W}$. 
\end{lem}
\begin{proof}
It is clear that $F\in \mca{W}$ satisfies the conditions. Conversely let $F$ 
satisfy the conditions. By the algebraic equivalence and $h^0 (\mca{O} (F))=1$, 
$F$ is a translate of some pullback 
of theta divisor: $F=\varepsilon^* \Theta_{\beta}+ \gamma$, $\gamma\in J(C)$. 
For any $x\in \varepsilon^* \Theta_{\beta}$, $-x\in \varepsilon^* \Theta_{\beta}$
and the former condition implies $-(-x+\gamma)\in F$, 
hence $x\in  \varepsilon^* \Theta_{\beta}+ 2\gamma$. Thus $2\gamma =0$. 

The last assertion follows from Proposition \ref{11th}. In fact,
since $(\varepsilon^* \Theta_{\beta}, \eta^* \Theta_{\beta})=(\varepsilon^* \Theta_{\beta}, 3\Theta_{\beta}-\varepsilon^* \Theta_{\beta})=7$, 
there are 
$32$ curves in $\mca{W}$.  
Let $F\in \mca{W}$. Then by the conditions, it corresponds to the unique twisted cubic curve 
in Proposition \ref{11th}. They are determined by the six nodes of $\overline{X}$. 
Hence $F$ can be recovered from the Weber hexad.
\end{proof}

Now we arrive at the following theorem.
\begin{thm}\label{main}
Let $C$ be a curve of genus two and $(X,W)$ its Jacobian Kummer surface and a Weber hexad on it.
Then the following conditions are equivalent.
\begin{enumerate}
\item The Hessian model $X_W$ acquires the $11$th node.
\item The Hutchinson-Weber involution $\sigma_W$ degenerates in the sense it acquires fixed loci.
\item The unique twisted cubic $\overline{E}$ passing through the nodes 
$\{n_{\alpha}\}_{\alpha\in W}$ of $\overline{X}$
lies on the Kummer quartic surface $\overline{X}$. (Here
the strict transform $E\subset X$ satisfies the relations in Proposition \ref{11th}.)
\item The Jacobian $J(C)$ is a Comessatti surface and one curve $\Xi$ among the 
set $\mca{W}$, (\ref{32}),
passes through the $2$-torsion points corresponding to $W$.
\item In the double plane model (Proposition \ref{H-N}) projected from 
one node $n_{w_0}\ (w_0\in W)$, there exists an additional conic $E'\subset \bb{P}^2$ which passes
through the vertices of the pentagon formed by five images of $\{n_w; w\in W-\{w_0\}\}$ 
and tangent to the remaining branch line.
For example, when $W=\{0,12,23,34,45,51\}$ and $w_0=0$ as in Proposition \ref{H-N}, 
then the pentagon is formed by $l_1, \cdots, l_5$ and the last line is $l_6$. 
\end{enumerate}
\end{thm}
\begin{proof}
(1)$\Leftrightarrow$(2) follows from Proposition \ref{41}.
(2)$\Leftrightarrow$(3) follows from Proposition \ref{11th}.

(3)$\Rightarrow$(4): The inverse image $\Xi\subset J(C)$ of $\overline{E}$ is a genus two curve
with $(\Xi, \Theta)=3$ since $\overline{E}$ is a cubic curve. Then the endomorphism $\varphi$ 
corresponding to the divisor $\Xi$ in the isomorphism (\ref{symm}) (which holds in general)
satisfies the relation $\varphi^2 -3
\varphi +1=0$, hence $J(C)$ is Comessatti. By construction $\Xi$ corresponds to some element
in $\mca{W}$, by Lemma \ref{cha}.
(4)$\Rightarrow$(3) is already mentioned in the proof of Lemma \ref{cha}.

(3)$\Leftrightarrow$(5): These correspond to each other 
as $E'$ is the image of $E$ by the projection $X\rightarrow \bb{P}^2$.
\end{proof}
\begin{rem}
The proof of Humbert's theorem in \cite{vdg} covers (3)$\Leftrightarrow$
(4)$\Leftrightarrow$(5) except for mentioning Weber hexads.
\end{rem}
\begin{prop}\label{six}
If $W$ and $W'$ are equivalent as dual six, then the conditions of previous theorem
for $W$ and $W'$ are equivalent.
\end{prop}
\begin{proof}
The condition (3) is the easiest translated into this proposition.
By using Proposition \ref{switch}, we can see easily that the images 
$\sigma_{\alpha_0} (E), \sigma_{\beta_0}(E)$ (by translations and switches) 
satisfy the conditions in Proposition \ref{11th} for other equivalent $W'$s. 
\end{proof}

\section{Periods}\label{mod}

General HW involutions $\sigma_W$ are fixed-point-free, hence they determine Enriques surfaces.
The moduli of Enriques surfaces obtained in this way is isomorphic to
an open set of the moduli 
of pairs $(X,W)$ where $X$ is a Jacobian Kummer surface and $W$ is a Weber hexad, 
considered modulo equivalence as dual six. By what we have studied, we can describe the 
boundary divisor consisting of Kummer surfaces of Comessatti Jacobians explicitly.\\

First we recall the periods of Jacobian Kummer surfaces.
We fix a lattice $T=U(2)\oplus U(2) \oplus \bracket{-4}$, which is isomorphic to 
the transcendental lattice of Picard-general Jacobian Kummer surfaces.
Recall that $T$ has a unique embedding into a $K3$ lattice $L_{K3}$. We formally take 
a $\bb{Z}$-generator $\{N_{\alpha}, T_{\beta}\}$ of the orthogonal complement $NS$ of $T$
analogous to that in Section \ref{JK}. Let $\Phi=\sum (N_{\alpha}+T_{\beta})/4\in L_{K3}$. 
Under these notation, we have the following criterion.
\begin{prop}(\cite[Theorem 6.3]{nikulin74}) 
Let $X$ be a $K3$ surface. Then $X$ is isomorphic to a Jacobian Kummer surface 
if and only if there exists a marking $H^2 (X, \bb{Z})\stackrel{\sim}{\rightarrow} L_{K3}$ 
inducing an embedding $T_X \subset T$ such that: under this marking,
there exists no $(-2)$-element $E$ in $NS(X)$ which is orthogonal to $\Phi$. 
\end{prop}
Let us compute the obstruction $E$. 
We put $E=E_{NS}+E_T$ according to the decomposition $L_{K3, \bb{Q}}= NS_{\bb{Q}} 
\oplus T_{\bb{Q}}$. 
After some computation, we obtain $E_{NS}= \pm H/4 \pm (\sum_{\alpha\in R} N_{\alpha})/2$,
where $R$ is a Rosenhain tetrad. Correspondingly we have $(E_T^2)=-1/4$.
Conversely, for any $E_T\in T^*$ with $(E_T^2)=-1/4$, it is easy to see that there exists
an element $E_{NS}\in NS^*$ such that 
$E_{NS}+ E_T\in L_{K3}$ and $((E_{NS}+ E_T)^2)=-2$. (In fact any $(1/4)$-element in the discriminant group $NS^*/NS$ 
corresponds to a patching element of a switch of an even theta characteristic, 
\cite[Section 5]{ohashi09}.)
Let 
\[ \mca{E}= \{e \in T^* \mid (e^2)=-1/4 \}\]
and $H_e\subset T_{\bb{C}}$ be the hyperplane orthogonal to $e\in \mca{E}$. 
Since $T$ has a unique primitive embedding into $L_{K3}$, we obtain 
\begin{prop}\label{JKmoduli}
The moduli space $\mca{JKS}$ of Jacobian Kummer surfaces is isomorphic to the period domain
\[\mca{D}(T) -\cup_{e\in \mca{E}} H_e, \text{ where }
\mca{D}(T)
= \{ [\omega ]\in \bb{P} (T_{\bb{C}})\mid (\omega^2)=0, 
(\omega,\overline{\omega})>0 \}\]
divided by the arithmetic group $O(T)$. 
\end{prop}
We remark that we can show $O(T)$ acts 
on $\mca{E}$ transitively, hence the divisor removed is irreducible. The proof
is the same as that of Lemma \ref{tr} below.\\

Next we consider the Weber hexads.
For the time being, suppose that $NS, T$ are identified with the Neron-Severi $NS(X)$ 
and the transcendental lattice $T_X$ of a Picard-general surface $X$. 
Recall that the discriminant group $T_X^*/T_X$ has exactly $6$ cyclic subgroups $C_W$ 
of order $4$, 
whose generators have the norm $(3/4)$ mod $2\bb{Z}$. 
These subgroups are exactly those arising as the patching subgroups of 
HW involutions, \cite[Section 7]{ohashi09}. In other words they are one-to-one to the dual six.
The correspondence is given by 
\begin{equation}\label{WH}
\text{(the class of) $W$} \leftrightarrow C_W=
\biggl\langle \frac{3}{4}H-\frac{1}{2}\biggl( \sum_{\alpha\in W} N_{\alpha}\biggr) 
\biggr\rangle
\subset NS(X)^*/NS(X),
\end{equation}
via the sign-reversing isometry $NS(X)^*/NS(X)\simeq T_X^*/T_X$. 

We return to the general situation.
Let us fix one subgroup $C_0\subset T^*/T$ as above once and for all. 
\begin{dfn}
For a pair $(X,W)$ of a Jacobian Kummer surface and an equivalence class of Weber hexads,
a {\em{marking}} $\phi \colon H^2 (X, \bb{Z})\stackrel{\sim}{\rightarrow} L_{K3}$ 
is an isometry satisfying the following conditions:
\begin{itemize}
\item (Lattice polarization): $\phi^{-1} (NS)$ coincides with the sublattice of $NS(X)$ 
generated by the $(16)_6$ configuration. We denote this sublattice by $NS(X)'$.
\item The subgroup $C_W\subset NS(X)'^*/NS(X)'$ defined by (\ref{WH}) 
corresponds to $\phi^{-1} C_0$ via $NS(X)'^* / NS(X)' \simeq \phi^{-1} (T^*/T)$. 
\end{itemize}
\end{dfn}
Let $\Gamma$ be the subgroup of $O(T)$ whose induced action on $T^*/T$ 
stabilizes $C_0$. Clearly $\Gamma$ acts on the set of markings of a pair $(X,W)$ 
and the moduli space of $(X,W)$ is given by restricting the arithmetic group to $\Gamma$.
\begin{prop}
The moduli space $\mca{JKS}_W$ of 
Jacobian Kummer surfaces equipped with a Weber hexad, considered modulo 
the equivalence as dual six, is isomorphic to the period domain 
$\mca{D} (T)- \cup_{e\in \mca{E}} H_e$ divided by the arithmetic subgroup $\Gamma$. 
\end{prop}
\erase{If you are not sure, check $\{ (X,W, \phi)\} \stackrel{\sim}{\rightarrow}
\mca{D} (T)- \cup_{e\in \mca{E}} H_e$.}
By \cite[Lemma 3.3]{ohashi09}
the natural projection $\mca{JKS}_W\rightarrow \mca{JKS}$ is $6:1$ and corresponds to 
the forgetful map $(X, W)\mapsto X$.

Let us compute the locus of degenerate HW involutions.
The HW involution $\sigma_W$ degenerates if and only if 
there exists a curve $E\in H^2 (X,\bb{Z})$ 
as in Proposition \ref{11th}. From the relations there, the
element $E=E_{\phi^{-1}(NS)}+E_{\phi^{-1}(T)}$ satisfies 
$E_{\phi^{-1}(NS)}=(3/4)H-(\sum_{\alpha\in W} N_{\alpha})/2$, where $N_{\alpha}$ is 
the $(16)_6$-configuration on $X$.
Hence $e=\phi(E_{\phi^{-1}(T)})$ satisfies the conditions 
\begin{equation}\label{obs}
 e \in T^*, (e^2)=-5/4, \text{$e$ generates $C_0$ in $T^*/T$}.
\end{equation}
Conversely if such an element $e$ exists and orthogonal to the period 
under a marking (as a pair $(X,W)$), then by \cite[Section 7]{ohashi09}
we obtain a $(-2)$-element $E\in NS(X)$ satisfying the numerical conditions in Proposition 
\ref{11th}. By Riemann-Roch, nef and big property of $L$ and the Cauchy-Schwarz inequality,
$E$ is a sum of $(-2)$-{\em{curves}} and then Proposition \ref{11th} shows that $E$ 
is a class of irreducible $(-2)$-curve. Thus the above condition 
is also sufficient for the degeneration.

Let 
\[\mca{E}'= \{e\in T^* \mid (e^2)=-5/4\}.\]
\begin{lem}\label{tr}
$O(T)$ acts on $\mca{E}'$ transitively.
\end{lem}
\begin{proof} 
Instead of $e\in T^*$ we consider the element $f=4e\in T$ which is primitive, 
$(f^2)=-20$ and $(f, T)=4\bb{Z}$. Clearly the transitivity for $e$ follows from 
that for $f$. 

The bilinear form of the lattice $T$ is always even, hence 
the problem reduces to that in the lattice $T(1/2)=U^2\oplus \bracket{-2}$. Because 
it contains two hyperbolic planes, \cite[Proposition 3.7.3]{scattone} concludes the proof.
\end{proof}
\begin{cor}\label{tr2}
$\Gamma$ acts on the set 
\[\mca{E}'_0 = \{e\in T^* \mid (e^2)=-5/4, \text{ and $e$ generates $C_0$ in $T^*/T$}\}\]
transitively.
\end{cor}
\begin{proof}
This follows from the lemma by definition.
\end{proof}
Hence we obtain 
\begin{thm}\label{moduli}
The moduli space $\mca{CJKS}=\{(X,W)\}$ of Comessatti Jacobian Kummer surfaces 
which satisfy the conditions of Theorem \ref{main} is isomorphic to the quotient 
$(\cup_{e\in \mca{E}'_0} H_e - \cup_{e\in \mca{E}} H_e)/\Gamma$, which 
is in fact irreducible by the previous corollary. 

The moduli space of Enriques surfaces obtained by HW involutions is
given by 
\[(\mca{D}(T) -\cup_{e\in \mca{E}} H_e - \cup_{e\in \mca{E}'_0} H_e)/\Gamma.\]
\end{thm}

\section{An application to the patching subgroups}

This section aims at giving a better way of understanding \cite[Proposition 7.3]{ohashi09}
and reproving it. We hope there are other cases to which our ideas will be applicable.

We fix a Weber hexad $W$ once and for all.
First we recall the situation of \cite[Section 7]{ohashi09}. 
Let $X_1$ be a Picard-general Jacobian Kummer surface 
and $\sigma_{W,1}$ be the HW involution. 
The problem is to determine the patching subgroup 
$\Gamma_{\sigma_{W,1}}$ which was defined in \cite[Definition 2.2]{ohashi09}. 
To this end, we can use the degeneration of HW involutions we have studied in this paper. 

We consider a one-dimensional smooth family of Jacobian Kummer surfaces $f:X\rightarrow \Delta$ 
(in what follows the letters $X$ and $X_W$, $N_{\alpha}$, $\sigma_W$, 
etc. represents a {\em{family}}
of surfaces, divisors, automorphisms, etc.) and 
its associated Hessian model $X_W \rightarrow \Delta$ with fibers
\[X_{W,t} \colon \sum s_i=\sum \lambda_i(t)/s_i =0, \quad t\in \Delta,\]
where $\Delta$ is a small disk.
We can assume that the Hessian model $X_{W,1}$ of $X_1$ 
appears as some fiber (over $t=1$, say) and the central fiber 
$X_{W,0}$ has eleventh node $p$ while the other fibers have exactly ten nodes.
The HW involution $\sigma_W=\{\sigma_{W,t}\}_{t\in \Delta}$ 
acts on $X_W$ birationally and fiberwisely. 
Blowing up the ten (families of) nodes corresponding to $N_{\alpha}\ (\alpha\in J(C)_2-W)$,
we obtain the family $\widetilde{X}_W
\rightarrow \Delta$ whose fibers are smooth for $t\in \Delta -\{ 0 \}$
and $\widetilde{X}_{W,0}$ has one node $p$. 
This is the same situation as in \cite[Section 7]{burns-rapoport}.
\[X \stackrel{\pi}{\rightarrow} \widetilde{X}_W \rightarrow \Delta ,\text{ ($\pi$: small)}.\]

On $\widetilde{X}_W$, $\sigma_W$ acts biregularly and fiberwisely.
Let us denote by $\Gamma\subset X\times_{\Delta} X$ the graph of $\sigma_W$;
since $\pi$ is an isomorphism over $\Delta^*=\Delta -\{0\}$, $\Gamma$ is just
the closure of $\Gamma\mid_{\Delta^*}$. Let $\Gamma_t \subset X_t\times X_t$ 
be the fiber of $\Gamma\rightarrow \Delta$. We can think of
$\Gamma_{0}=\lim_{t\rightarrow 0} \Gamma_t$ as the limit in the Barlet space 
of $X\times_{\Delta} X$ as in \cite[VIII, Lemma 10.3]{BHPV}, \cite[Theorem 2,]{burns-rapoport}.
Hence the induced map on cohomology 
\[ [\Gamma_0]_* :H^2(X_0, \bb{Z}) \rightarrow H^2(X_0, \bb{Z}) \]
is the same as 
that of $[\Gamma_1]_*: H^2(X_1, \bb{Z}) \rightarrow H^2(X_1, \bb{Z})$ under the 
obvious trivialization of the local system $R^2 f_* \bb{Z}_X$. 

Clearly $\Gamma_t$ is the graph of HW involutions $\sigma_{W,t}$ 
for $t$ other than $0$.
But the point is that $\Gamma_0$ does not give the graph of HW involution 
$\sigma_{W,0}$, because $\sigma_{W,0}$ has fixed points and therefore its action on the cohomology 
cannot be the same as other $\sigma_{W,t}$. By \cite[VIII, Proposition 10.5]{BHPV}
$\Gamma_0$ is of the form $\Lambda_0+ E\times E$, where $\Lambda_0$ is the graph of $\sigma_{W,0}$
and $E\subset X_0$ is the fixed curve of $\sigma_{W,0}$, namely exceptional curve for $\pi$. 
Therefore the induced map is of the form 
\[
[\Gamma_0]_* = [\Lambda_0+ E\times E]_* : x \mapsto \sigma^*_{W,0} (x) + (x,E)E. 
\]
Since $E$ is the fixed curve of $\sigma_{W,0}$, it follows 
that $[\Gamma_0]_* = \sigma^*_{W,0} \circ r_E = r_E \circ \sigma^*_{W,0}$ where $r_E$ is 
the reflection in $E$. 

Let us return to the computation of the patching subgroup $\Gamma_{\sigma_{W,1}}$. 
We have seen that the action of $\sigma_{W,1}$ on the cohomology is the same 
as 
\[
\sigma_{W,1}^* = [\Gamma_1 ]_* = [\Gamma_0 ]_* = \sigma^*_{W,0} \circ r_E,
\]
where $\sigma_{W,0}$ is the degenerate HW involution. 
In particular $\sigma_{W,1}^* (E)=-E$ (this is not a contradiction since the cycle $E$ 
is not an algebraic cycle at $t=1$). Let us write $E$ as $E_{NS}+E_T$ according to 
the orthgonal decomposition over the rationals $H^2 (X_1, \bb{Q}) = NS(X_1)_{\bb{Q}}
\oplus T_{X_1,\bb{Q}}$. Using the relations in Proposition \ref{11th}, 
it is easy to see that $E_{NS}= (3/4)H_1-(\sum_{\alpha\in W} N_{\alpha ,1})/2$.
By the definition of the patching subgroup, $E_{NS}$ is the patching element.
Since we know that $\Gamma_{\sigma_{W,1}}$ is of order $4$, it is generated by the class of 
$E_{NS}$.


\begin{thebibliography}{ABCDEFG}
\bibitem{atlas}
	{J. H. Conway, R. T. Curtis, S. P. Norton, R. A. Parker and R. A. Wilson},
	{ATLAS of finite groups},
	{Clarendon Press, Oxford, 1985}.
\bibitem{BHPV}
	{W. Barth, K. Hulek, C. Peters and A. Van de Ven},
	{Compact Complex Surfaces (Second Enlarged edition)},
	{Erg. der Math. und ihrer Grenzgebiete, 3. Folge, Band 4.},
	{Springer, 2004}.
\bibitem{birkenhake-lange}
	{C. Birkenhake and H. Lange},
	{Complex Abelian Varieties}, 
	{Grundlehren der mathematischen Wissenschaften, Volume 302},
	{Springer, 1992}.
\bibitem{burns-rapoport}
	{D. Burns and M. Rapoport},
	{On the Torelli problem for K\"{a}hlerian $K3$ surfaces},
	{Ann. Sc. ENS.},
	\textbf{8} (1975), 235-274.
\bibitem{dardanelli-vg}
	{E. Dardanelli and B. van Geemen},
	{Hessians and the moduli space of cubic surfaces},
	{Contemp. Math.}, \textbf{422},
	{Amer. Math. Soc., Providence, RI, (2007).}
\bibitem{dolgachev-keum}
	{I. V. Dolgachev and J. H. Keum},
	{Birational automorphisms of quartic hessian surfaces},
	{Trans. Amer. Math. Soc.},
	\textbf{354} (2002), 3031-3057.
\bibitem{GH}
	{P. Griffiths and J. Harris}, 
	{Principles of algebraic geometry},
	{Wiley Classics Library},
	{John Wiley \& Sons, Inc., New York, 1994.}
\bibitem{hulek-lange}
	{K. Hulek and H. Lange},
	{The Hilbert modular surface for the ideal $(\sqrt{5})$ and the 
	Horrocks-Mumford bundle},
	{Math. Z.},
	\textbf{198} (1988), 95-116.
\bibitem{mumford}
	{D. Mumford},
	{Abelian Varieties},
	{Oxford University Press, 1970.}
\bibitem{nikulin74}
	{V. V. Nikulin},
	{An analogue of the Torelli theorem for Kummer surfaces of Jacobians},
	{Izv. Akad. Nauk SSSR, Ser. Mat.},
	\textbf{38} (1974), 21-41.
\bibitem{ohashi09}
	{H. Ohashi},
	{Enriques surfaces covered by Jacobian Kummer surfaces},
	{Nagoya Math. J.},
	\textbf{195} (2009), 165-186.	
\bibitem{scattone}
	{F. Scattone},
	{On the compactification of moduli spaces for algebraic $K3$ surfaces},
	{Mem. Amer. Math. Soc.},
	\textbf{70} (1987), no. 374, x+86 pp. 
\bibitem{vdg}
	{G. van der Geer},
	{Hilbert modular surfaces}, 
	{Ergebnisse der Mathematik und ihrer Grenzgebiete 3},
   \textbf{16} {Springer-Verlag, Berlin}, 1988.
\end{thebibliography}
\end{document}